\documentclass[12pt]{article}
\usepackage{amsmath,amssymb,amsthm}
\usepackage[colorlinks,citecolor=blue]{hyperref}
\usepackage[utf8]{inputenc}
\usepackage{tikz}
\usetikzlibrary{shapes.geometric,calc}
\usepackage[all]{xy}

\newtheorem{theorem}{Theorem}[section] 
\newtheorem{proposition}[theorem]{Proposition} 
 
\newtheorem{corollary}[theorem]{Corollary} 
\newtheorem{lemma}[theorem]{Lemma}

\theoremstyle{definition}
\newtheorem{definition}[theorem]{Definition}
\newtheorem{remark}[theorem]{Remark}

\newcommand{\ir}{\mathrm{ir}}
\newcommand{\id}{\mathrm{dr}}

\newcommand{\n}{{\color{red} n}}

\newcommand{\Tam}{\operatorname{Tam}}
\newcommand{\Ri}{\operatorname{Ri}}
\newcommand{\Fa}{\operatorname{Fa}}
\newcommand{\NC}{\operatorname{NC}}

\title{Exceptional and modern intervals of the Tamari lattice}
\author{Baptiste Rognerud\footnote{Cet auteur a b\'en\'efici\'e d'un financement de l'IDEX BMM/PN/AM/N\textsuperscript{o} 2016-096c.}}
\date{}

\usepackage{array}
\newcolumntype{P}[1]{>{\centering\arraybackslash}p{#1}}

\begin{document}

\maketitle


\begin{abstract}
In this article we use the theory of interval-posets recently introduced by Châtel and Pons in order to describe some interesting families of intervals in the Tamari lattices. These families are defined as interval-posets avoiding specific configurations. At first, we consider what we call \emph{exceptional} interval-posets and show that they correspond to the intervals which are obtained as images of noncrossing trees in the Dendriform operad. We also show that the exceptional intervals are exactly the intervals of the Tamari lattice induced by intervals in the poset of noncrossing partitions. In the second part we introduce the notion of \emph{modern} and \emph{infinitely modern} interval-posets. We show that the modern intervals are in bijection with the new intervals of the Tamari lattice in the sense of Chapoton. We deduce an intrinsic characterization of the new intervals in the Tamari lattice. Finally, we consider the family of what we call infinitely modern intervals and we  we prove that there are as many infinitely modern interval-posets of size $n$ as there are ternary trees with $n$ inner vertices. 
 \end{abstract}
\section{Introduction}

The family of the Tamari lattices is extremely rich from the point of view of combinatorial algebra. The Tamari lattices have two interpretations as posets of type $A$. First, it is isomorphic to the poset of tilting modules over a linearly oriented quiver of type $A$ (See \cite{happel_tilting} for more details. It seems that part of this was already observed by Gabriel \cite{gabriel}). On the other hand, Tamari lattices are part of the Cambrian lattices of type $A$ (See \cite{reading} for more details). Finally let $\mathrm{DW}_{n}$ be the distributive lattice of upper ideals in the poset of positive roots of the root system of type $A_{n-1}$. Then, the Tamari lattice of size $n$ is conjecturally deeply related to $\mathrm{DW}_{n}$ (See Conjecture $5.3$ of \cite{chapoton_modules} for more details). 

As another intriguing feature of this lattice, we have its poset of intervals. It was proved by Chapoton that there is a beautiful formula for the number of intervals in the Tamari lattice.
\[
\hbox{ Number of intervals in $\Tam_n = $} \frac{2(4n+1)!}{(n+1)!(3n+2)!}. 
\] 
It is remarkable that this formula has such a simple factorized form. More recently, in \cite{chapoton_interval}, Chapoton associated to any finite poset $P$ a polynomial in $4$ variables that enumerate the intervals of $P$ and he proved that the polynomial of the Tamari lattice has a very particular behavior (this particular behavior is not shared with generic posets). 

In this article, we continue to investigate the set of intervals of the Tamari lattices. We use the theory of interval-posets introduced by Châtel and Pons in \cite{pons_chatel} in order to study two families of intervals in the Tamari lattice. In terms of intervals, these families seem to have a rather complicated description. However,  they have a very simple description in terms of interval-posets avoiding specific configurations.

In the first part of the article we consider the family appearing as images of noncrossing trees in the dendriform operad. These objects were introduced by Chapoton in \cite{chapoton_mould}, and it was proved in \cite{chapoton_etal_mould} that they are intervals in the Tamari lattice. In Theorem \ref{image_noncrossing}, we complete this result by giving a precise description of these intervals in terms of interval-posets. By construction, they are in bijection with the noncrossing trees. In particular in the Tamari lattice of size $n$, there are $\frac{1}{2n+1} {{3n}\choose{n}}$ such intervals. We call them \emph{exceptional} because they are also in bijection with the set of exceptional sequences (up to an equivalence relation) in the bounded derived category of a linearly oriented quiver of type $A$ (See \cite{araya} and Section $3$ of \cite{chapoton_stokes} for more informations). We would need to introduce too many algebraic objects to really explain what we have in mind here, but we expect this relation with the exceptional sequences to be much more than a bijection. 

At an elementary level, the exceptional intervals turn out to have another nice description in terms of noncrossing partitions. It is well-known that the Tamari lattice is a refinement of the poset of noncrossing partitions. More precisely, if $\Tam_n$ denotes the Tamari lattice of size $n$ and $\mathrm{NC}_n$ denotes the poset of noncrossing partitions, then there is a bijection $\phi : \mathrm{NC}_n \to \Tam_n$ which is a morphism of posets. In Theorem \ref{theo_nc}, we prove that an interval of $\Tam_n$ is of the form $[\phi(\pi_1),\phi(\pi_2)]$ for an interval $[\pi_1,\pi_2]$ of noncrossing partitions if and only if it is exceptional. 

In the second part of the article, we consider the family of new intervals of the Tamari lattices. It was shown by the first author that there is a structure of operad on the set of intervals of the Tamari lattice (see \cite{chapoton_interval} for more details). The new intervals are exactly the intervals that cannot be obtained as compositions of smaller intervals. There is also a nice formula for the number of such intervals:

\[ \hbox{Number of new intervals in $\Tam_n$ }= 3 \cdot \frac{2^{n-2}(2n-2) !}{(n-1) ! (n+1) !}.\]

In Section $4$, we find the description of the interval-poset corresponding to a new interval and we deduce an intrinsic characterization of these intervals. Our main tool is what we call the rise of an interval-poset. This operation increases the size of an interval-poset by $1$, and shifts by $1$ all the increasing relations of the poset. After shifting the increasing relations by $1$, the result is not necessarily a poset since the new increasing relations may contradict the decreasing ones. We introduce the family of modern interval-posets and show that they are exactly the interval-posets for which the rise is also an interval-poset. Then, we prove that an interval is new if and only if its interval-poset is the rise of a modern interval-poset. In terms of interval, the rise sends an interval $[S_1,T_1]$ to the interval $[S,T]$ where $S$ (resp. $T$) is obtained by grafting the root of $S_1$ (resp. $T_1$) on the first (resp. second) leaf of $Y$ the binary tree of size $1$.

In the last section, we consider the interval-posets for which all the successive risings are interval-posets. We call them \emph{infinitely modern}. It seems that this family of intervals have not been considered before. Using a double statistic on the set of interval-posets, we recover the triangular decomposition of the Fuss-Catalan number $\frac{1}{2n+1} {{3n}\choose{n}}$ introduced by Aval in \cite{aval}. As corollary, we prove in Theorem \ref{theo_inf} that there are as many infinitely modern interval-posets of size $n$ as there are ternary trees with $n$ inner vertices.

\paragraph{Acknowledgement}
This work was done when I was a postdoc at the University of Strasbourg and I am grateful to Frédéric Chapton for introducing me to this subject, for his support, his comments and the many things he taught me. I am also grateful to Camille Combe for the many discussions about the last part of this article.

\section{Interval-posets, intervals of the Tamari lattices and conventions}\label{section1}

In this section we recall the construction of \emph{interval-posets} of Châtel and Pons introduced in \cite{pons_chatel} and recall that they are in bijection with the intervals of the Tamari lattice. One should note that this bijection is not canonical. More precisely, it depends on the various choices that one has to make in order to define the Tamari lattices as partial orders on sets of binary trees. This is why we start by carefully stating our conventions.

Let $n\in \mathbb{N}$. A (planar) binary tree of size $n$ is a graph embedded in the plane which is a tree, has $n$ vertices with valence $3$, $n+ 2$ vertices with valence $1$ and a distinguished univalent vertex called the \emph{root}. The other vertices of valence $1$ are called the \emph{leaves} of the tree. For the rest of the paper, when we speak about vertices of the tree, we have in mind the trivalent vertices. The planar binary trees are pictured with their root at the bottom and their leaves at the top.

With this fixed convention, we can speak about \emph{left} and \emph{right} sons (or children) of a vertex of a binary tree $T$. For us the son of a vertex is connected to his father by a single edge, if there is more than one edge we speak about a \emph{descendant}. If $v$ is a vertex of $T$, we let $T_1$ (resp. $T_2$) be the subtree with root the left son (resp. right son) of $v$. We say that $T_1$ (resp. $T_2$) is the left subtree (resp. right subtree) of $v$. 

Let $\Tam_n$ be the set of all binary trees with $n$ vertices. It is well-known that the cardinality of this set is the Catalan number $c_n = \frac{ 1 }{n+1} {{2n}\choose{n}}$. 

There is a partial order relation on $\Tam_n$ which was introduced by Tamari in \cite{tamari}. It is defined as the transitive closure of the following covering relations. A tree $T$ is covered by a tree $S$ if they only differ in some neighborhood of an edge by replacing the configuration $\begin{tikzpicture}[scale = 0.1]
        \draw (0,0)--(2,2); \draw (0,0)--(-2,2); \draw (-1,1)--(0,2);
\end{tikzpicture}$ in $T$ by the configuration  $\begin{tikzpicture}[scale = 0.1]
        \draw (0,0)--(2,2); \draw (0,0)--(-2,2); \draw (1,1)--(0,2);
\end{tikzpicture}$ in $S$. The poset $\Tam_n$ is known to be a lattice. 

A binary \emph{search tree} is a binary tree labelled by integers such that if a vertex $x$ is labelled by $k$, then the vertices of the left subtree (resp. right subtree) of $x$ are labelled by integers  less than or equal (resp. superior) to $k$. 

if $T$ is a binary tree with $n$ vertices, there is a unique labelling of the vertices by each of the integers $1,2,\cdots, n$ that makes it a binary search tree. This procedure is sometimes called the \emph{in-order traversal} of the tree. The insertion procedure is recursive. Starting at the root of the tree $T$, the algorithm is the following. 

\begin{verbatim}
1. Traverse the left subtree, i.e., call in-order(left-subtree)
2. Visit the root.
3. Traverse the right subtree, i.e., call in-order(right-subtree)
\end{verbatim}
The first vertex visited by the algorithm is labeled by $1$, the second by $2$ and so on. See figure \ref{fig_inorder} for an example. Since this labeling is canonical, we will allow ourself to identify vertices with their label.

Using this labeling, a binary tree $T$ with $n$ vertices induces a partial ordered relation $\lhd$ on the set $\{1,\cdots, n\}$ by setting $i \lhd j$ if and only if the vertex labelled by $i$ is in the subtree with root $j$. 

When $(P,\lhd)$ is a partial order on the set $\{1,\cdots, n\}$, one can use the natural total ordering of the integers $1,\cdots, n$ that we denote by $<$ to split the relations $\lhd$ in two families. Let $1 \leqslant a < b \leqslant n$ be two integers. If $a \lhd b$ we say that the relation is \emph{increasing}.  On the other hand, if $b\lhd a$, we say that the relation is \emph{decreasing}. We denote by $\mathrm{Dec}(P)$ and $\mathrm{Inc}(P)$ the set of decreasing and increasing relations of $P$. 

There is a particularly nice way to draw such a poset $(P,\lhd)$. If a relation $i \lhd j$ is increasing, draw a (red) arrow from $i$ to $j$ under the integers $i,i+1,\cdots, j$. If there is a decreasing relation $j \lhd i$ draw a (blue) arrow from $j$ to $i$ over the integers $j,j-1,\cdots, i$. See figure \ref{fig_inorder} for an example.

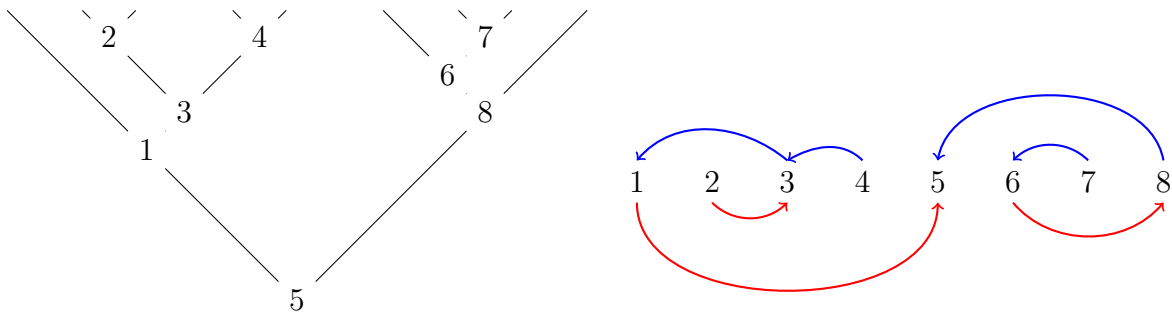
\begin{figure}[h]
\centering
\begin{tikzpicture}[scale =1]
\node (a) at (0,0) {5};
\node (c) at (-1.5,2.5) {3};
\node (b) at (-2,2) {1};
\node (d) at (-2.5,3.5) {2};
\node (e) at (-0.5, 3.5) {4};
\node (f) at (2.5, 2.5) {8};
\node (g) at (2,3) {6} ; 
\node (h) at (2.5,3.5) {7};
\node (1) at (-4,4) {}; \node (2) at (-3,4) {}; \node (3) at (-2,4) {}; \node (4) at (-1,4) {}; \node (5) at (0,4) {}; \node (6) at (1,4) {}; \node (7) at (2,4) {}; \node (8) at (3,4) {}; \node (9) at (4,4) {}; 
\draw (a)--(b)--(1)
	  (b)--(c)--(e) -- (5)
	  (e)--(4)
	  (c)--(d)--(2)
	  (d)--(3)
	 (a)--(f)--(9)
	 (f)--(g)--(6)
	 (g)--(h)--(7)
	 (h)--(8);
\end{tikzpicture}
\begin{tikzpicture}
\node (1) at (-4,4) {1}; \node (2) at (-3,4) {2}; \node (3) at (-2,4) {3}; \node (4) at (-1,4) {4}; \node (5) at (0,4) {5}; \node (6) at (1,4) {6}; \node (7) at (2,4) {7}; \node (8) at (3,4) {8};
\draw[red,->,below,thick] (2.south) to [out=-45,in=225] (3.south);
\draw[red,->,below,thick] (1.south) to  [out=-90,in=-90] (5.south);
\draw[red,->,below,thick] (6.south) to  [out=-50,in=230] (8.south);
\draw[blue,thick,->] (3.north) to  [out=140,in=50] (1.north);
\draw[blue,thick,->] (4.north) to  [out=135,in=30] (3.north);
\draw[blue,thick,->] (7.north) to  [out=135,in=45] (6.north);
\draw[blue,thick,->] (8.north) to  [out=100,in=80] (5.north);
\end{tikzpicture}
\caption{On the left, an example of the labeling of the vertices of a binary tree by calling the `in-order' algorithm. On the right, the poset induced by the tree.}\label{fig_inorder}
\end{figure}
Using this, we have a useful characterization due to Châtel, Pilaud and Pons \cite{chatel_pilaud_pons} of the partial order of the Tamari lattice in terms of increasing or decreasing relations.
\begin{proposition}\label{char_tamari}
Let $T_1$ and $T_2$ be two binary trees. Then $T_1 \leqslant T_2$ in the Tamari lattice if and only if $\mathrm{Dec}(T_1) \subseteq \mathrm{Dec}(T_2)$ if and only if $\mathrm{Inc}(T_2) \subseteq \mathrm{Inc}(T_1)$. 
\end{proposition}
\begin{proof}
See Proposition $40$ and Remark 52 of \cite{chatel_pilaud_pons}.
\end{proof}
\begin{definition}
An interval-poset $(P, \lhd) $ is a poset over the integers $1,\cdots, n$ such that 
\begin{enumerate}
\item If $a\lhd c$ and $a<c$, then for all integers $b$ such that $a<b<c$, we have $b\lhd c$.
\item If $c\lhd a$ and $a<c$, then for all integers $b$ such that $a<b<c$, we have $b\lhd a$. 
\end{enumerate}
\end{definition}
The conditions $(1)$ and $(2)$ of this definition will be referred as the \emph{interval-poset condition}. The integer $n$ in the definition is called the \emph{size} of the interval-poset. 
\begin{remark}
Let $(P,\lhd)$ be an interval-poset. If $x\lhd y$ is an increasing relation (resp. a decreasing relation), then by the interval-poset condition there is a relation $y-1\lhd y$ (resp. $x+1\lhd x$). The existence of such `small' relations will be crucial in most of our proofs on modern interval-posets. 
\end{remark}
\begin{theorem}[Châtel, Pons]\label{bij}
Let $n \in \mathbb{N}$. There is a bijection between the set of intervals in $\Tam_n$ and the set of interval-posets of size $n$. 
\end{theorem}
\begin{proof}
This is Theorem $2.8$ of \cite{pons_chatel}.

Since we need to use the explicit version of the theorem, let us recall the bijections. if $[S,T]$ is an interval in $\Tam_n$, we can construct an interval-poset as follows. The trees $S$ and $T$ can be seen as binary search trees and they induce two partial order relations $\lhd_S$ and $\lhd_T$. Let $P = \{1,2,\cdots, n\}$. There is a binary relation $\lhd$ on $P$ given by the disjoint union of the decreasing relations of $S$ and the increasing relations of $T$.  Then, it is proved in \cite{pons_chatel} that $(P,\lhd)$ is an interval-poset. 

Conversely, if $(P,\lhd)$ is an interval-poset of size $n$. Let $D$ be the poset obtained from $P$ by keeping only the decreasing relations of $P$. Similarly let $I$ be the poset obtained by keeping the increasing relations. By Lemma $2.5$ of \cite{pons_chatel}, the Hasse diagrams of these two posets are two forests. If we add a common root to the trees of each of these forests, we obtained two planar trees. Now, we produce binary trees starting from these planar trees.

For $I$ we recursively produce a binary tree $T$ by using the rule: right brother becomes right son and son becomes left son.

For $D$ we recursively produce a binary tree $S$ by using the rule: left brother becomes left son and son becomes right son.

The tree $S$ is smaller than $T$ for the order of the Tamari lattice, so we have an interval $[S,T]$.

These two correspondences are sometimes called the \emph{Knuth correspondences} or the \emph{natural correspondences} (see \cite{bruijn_morselt} or \cite{HPGT} for more details). 

It was proved in Theorem $2.8$ of \cite{pons_chatel} that these two constructions give two bijections inverse of each other. 
\end{proof}
Finally, we need a useful translation in the world of interval-poset of the usual left/right symmetry of trees.
\begin{lemma}\label{symmetry}
Let $[S,T]$ be an interval in  $\Tam_n$ and $P$ be its corresponding interval-poset. The interval-poset corresponding to the interval obtained by taking the left/right symmetry of $S$ and $T$ is the interval-poset $Q$ of size $n$ defined by $a \lhd_Q b \Leftrightarrow n+1 - a \lhd_{P} n+1 - b$. 
\end{lemma}
\section{Exceptional intervals of the Tamari lattice}

In \cite{chapoton_mould}, Chapoton introduced an operad $\mathbf{NCP}$ of \emph{noncrossing plants}. A non-crossing plant is a generalization of a noncrossing tree. Since we will not work with them, we refer the reader to the original article for a precise definition. We will only use that noncrossing trees are particular examples of noncrossing plants. It was proved that this operad (in the category of sets) is a sub-operad of $\mathbf{Dend}$, the Dendriform operad. Then, it was proved in \cite{chapoton_etal_mould} that the image of a noncrossing tree in $\mathbf{Dend}$ is of the form $\sum_{t\in I} t$ where $I$ is an interval in the Tamari lattice. An interval that appears as such image of a noncrossing tree is called \emph{exceptional}. In this section, we reprove and precise this result by giving an explicit description of the exceptional intervals of the Tamari lattice in terms of the interval-posets. Since they are in bijection with the noncrossing trees, the number of exceptional intervals in the Tamari lattice of size $n$ is $\frac{1}{2n+1}{{3n}\choose{n}}$. 

There is another well known family of intervals of the Tamari lattice of this size: it is classical that the Tamari order is a refinement of the usual partial ordering of the noncrossing partitions (see Section $2$ \cite{bernardi_bonichon} for more details). This implies that an interval in the poset of noncrossing partitions gives an interval in the Tamari lattice. By a result of Kreweras (\cite{kreweras}) or a bijection of Edelman (\cite{edelman}), the number of intervals of noncrossing partitions of size $n$ is $\frac{1}{2n+1}{{3n}\choose{n}}$. At the end of this section, we show that this family coincide with the family of exceptional intervals. 

\subsection{Exceptional intervals and noncrossing trees}

A \emph{noncrossing tree} in the regular $n+1$-gon is a set of edges between the vertices of the polygon with the following properties 
\begin{itemize}
\item edges do not cross pairwise,
\item any two vertices are connected by a sequence of edges,
\item There is no loop made of edges.
\end{itemize}
The boundary edges are allowed in the set. It is classical that the number of noncrossing trees in the regular $n+1$-gon is $\frac{1}{2n+1}{{3n}\choose{n}}$.

Given two noncrossing trees $f$ and $g$ in regular polygons and a side $i$ of the regular gon containing $f$, one can define the composition $f\circ_{i} g$ in the grafting of the polygons containing $f$ and $g$. This is defined as the union of the two trees, with some modifications along the grafting diagonal. If the diagonal is present in both $f$ and $g$, then it is kept in $f\circ_i g$. If it is present in exactly one of the two trees, then it is not kept in $f\circ_i g$. Otherwise, the result is not a noncrossing tree. One `denominator' diagonal is added and the result is a noncrossing plant. See Section $5.2$ of \cite{chapoton_mould} for more details.

It was shown in Paragraph $5.1$ of \cite{chapoton_etal_mould} that one can construct a poset from a noncrossing tree. Let us recall this construction.

Let $T$ be a noncrossing tree in a \emph{based} regular $n+1$-gon. Here by based we mean that we choose one side of the gon and call it the base. We can label the edges of the $n+1$-gon by assigning the number $0$ to the base, and then assigning the numbers $1$ to $n$ to the edges in a clockwise order. If an edge of $T$ is a boundary edge we assign to it the number of the boundary edge. Otherwise, the label of the edge of the noncrossing tree is the number of the unique open boundary edge that it separates from the base. Then, we set $i \lhd_{T} j$ if the edge $i$ is separated from the base by the edge $j$. An example is given in Figure \ref{fig1}.  
\begin{figure}
\begin{tikzpicture}[scale=1.3]
\node [draw, minimum size=5cm,regular polygon,regular polygon sides=12,rotate=180] (a) {};
\foreach \n in {2,3,11,12}{
   \pgfmathsetmacro\result{12 - (\n -1)}
    \node[below] at (a.side \n) {\pgfmathprintnumber{\result}};
}
\foreach \n in {5,6,...,9}{
   \pgfmathsetmacro\result{12 - (\n -1)}
    \node[above] at (a.side \n) {\pgfmathprintnumber{\result}};
}
\node[left] at (a.side 10) {3};
\node[right] at (a.side 4) {9};
\node[below] at (a.side 1) {0};
\draw[red,very thick] (a.corner 10) -- (a.corner 7);
\draw[red,very thick] (a.corner 10) -- (a.corner 9);
\draw[red,very thick] (a.corner 10) -- (a.corner 8);
\draw[red,very thick] (a.corner 11) -- (a.corner 10);
\draw[red,very thick] (a.corner 11) --(a.corner 12);
\draw[red,very thick] (a.corner 11) -- (a.corner 3);
\draw[red,very thick] (a.corner 3) -- (a.corner 2);
\draw[red,very thick] (a.corner 3) -- (a.corner 1);
\draw[red,very thick] (a.corner 7) -- (a.corner 4);
\draw[red,very thick] (a.corner 4) -- (a.corner 5);
\draw[red,very thick] (a.corner 7)-- (a.corner 6);
\end{tikzpicture}
\begin{tikzpicture}[scale=1.3]
\node [draw=none, minimum size=5cm,regular polygon,regular polygon sides=12,rotate=180] (a) {};
\node[left] at (a.side 10) {3};
\node[right] at (a.side 4) {9};
\draw[very thick] (a.corner 10) -- node[below] {6} (a.corner 7);
\draw[very thick] (a.corner 10) -- node[left] {4} (a.corner 9);
\draw[very thick] (a.corner 10) -- node[above] {5} (a.corner 8);
\draw[very thick] (a.corner 11) -- (a.corner 10);
\draw[very thick] (a.corner 11) -- node[below] {2} (a.corner 12);
\draw[very thick] (a.corner 11) --  node[above] {10} (a.corner 3);
\draw[very thick] (a.corner 3) --  node[below] {11} (a.corner 2);
\draw[very thick] (a.corner 3) -- node[above] {1} (a.corner 1);
\draw[very thick] (a.corner 7) --  node[below] {8} (a.corner 4);
\draw[very thick] (a.corner 4) --  (a.corner 5);
\draw[very thick] (a.corner 7)--  node[above] {7} (a.corner 6);
\end{tikzpicture}
\begin{tikzpicture}[scale=1.5]
 \node (a) at (0,0) {$1$};
 \node (b) at (-1,0.75) {$3$};
 \node (c) at (0,0.5) {$10$};
 \node (d) at (1,0.5) {$2$};
  \node (e) at (-0.5,1) {$6$};
  \node (f) at (0.5,1) {$8$};
  \node (g) at (-0.5,1.5) {$5$};
   \node (h) at (0.25,1.5) {$7$};
  \node (i) at (0.75,1.5) {$9$};
  \node (j) at (-0.5,2) {$4$};
  \node (k) at (1.5,0) {$11$};
  \draw (c) -- (b)
  	   (a) -- (c) -- (e) -- (g) -- (j)
	   (c)--(f)--(h)
	   (f)--(i)
	   (a)--(d);
 \end{tikzpicture}

\caption{On the left, an example of a noncrossing tree in a $12$-gon and on the middle the induced labelling of the noncrossing tree. On the right, the Hasse diagram of the corresponding poset where the maximal elements are $1$ and $11$. }\label{fig1}
\end{figure}
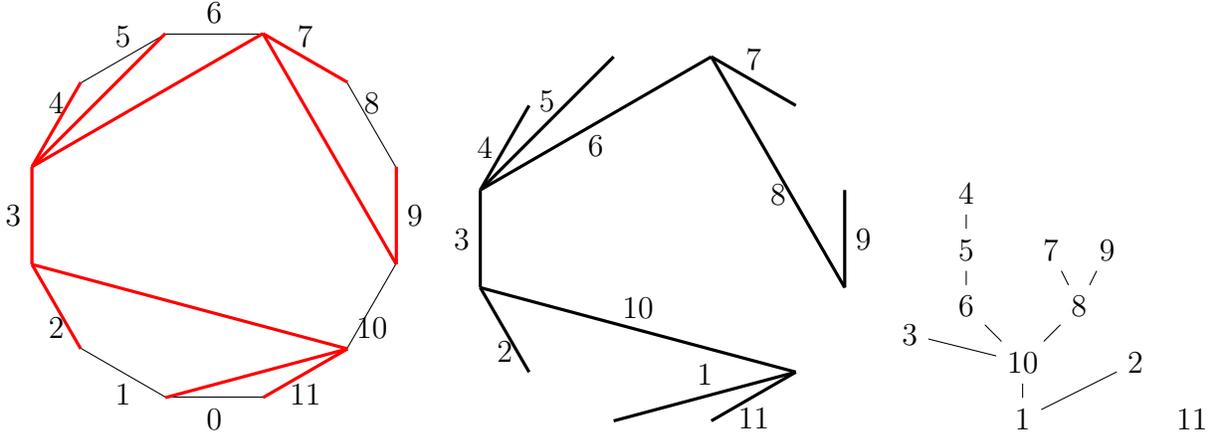

\begin{lemma}\label{lem1}
Let $T$ be a non-crossing tree in a based regular $n+1$-gon. Then, the poset $([1,n],\lhd_{T})$ is an interval-poset.
\end{lemma}
\begin{proof}
We label the boundary edges of a based regular $n+1$-gon as above. We use the notation $[i_1,i_2]$ where $i_1 \leqslant i_2$ for the edge that goes from the left side of the boundary edge $i_1$ to the right side of the boundary edge $i_2$. For example, in Figure \ref{fig1}, the edge with label $6$ corresponds to $[4,6]$ and the edge labelled by $4$ corresponds to $[4,4]$. Note that by construction of our labelling, the edge $[i_1,i_2]$ labelled by the number $i$ separates the boundary edge $i$ of the regular $n+1$-gon from the base. In particular this implies that $1\leqslant i_1 \leqslant i$ and $i\leqslant i_2 \leqslant n$.   

We want to check that the poset $(P_T,\lhd)$ is an interval poset. Let $0<i < j < k \leqslant n$ such that $i\lhd k$. This means that the edge $[i_1,i_2]$ labelled by $i$ is separated from the base by the edge $[k_1,k_2]$ labelled by $k$. Since, $k$ separates $i$ from the base and $T$ is a noncrossing tree, it is easy to check that the only possibility is to have 
 $$ k_1 \leqslant i_1 \leqslant i \leqslant i_2 \leqslant k \leqslant k_2. $$ Now, the boundary edge $j$ is between $i$ and $k$, so either it is before $i_2$ or after. Since $T$ is a noncrossing tree the edge $j$ cannot cross the edges $i$ and $k$. So it is easy to see that in the first case $k$ and $i$ separate $j$ from the base, and in the second case $k$ separates $j$ from the base. In particular, we have $j \lhd k$. See Figure \ref{fig2} for an illustration where the letter $j$ is used for the first case and the letter $J$ for the second. The case where $k\lhd i$ is similar and is illustrated in the right part of Figure \ref{fig2}.

\end{proof}
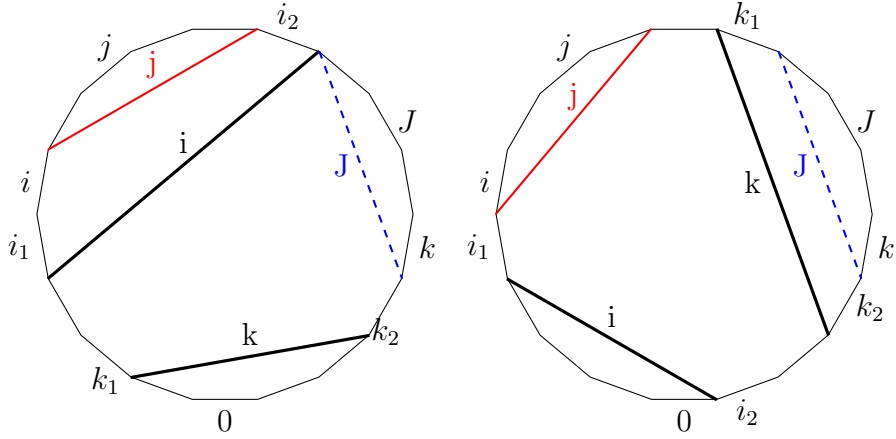
\begin{figure}
\centering \begin{tikzpicture}[scale=2]
\node [draw, minimum size=5cm,regular polygon,regular polygon sides=18] (a) {};
\node[below] at (a.side 10) {$0$};
\node[below] at (a.side 8) {$k_1$};
\node[left] at (a.side 5) {$i$};
\node[left] at (a.side 6) {$i_1$};
\node[above] at (a.side 18) {$i_2$};
\draw[very thick] (a.corner 7) -- node[above] {i} (a.corner 18);
\node[right] at (a.side 14) {$k$};
\node[below] at (a.side 13) {$k_2$};
\draw[very thick] (a.corner 9) -- node[above] {k} (a.corner 13);
\node[above] at (a.side 3) {$j$};
\draw[thick, red] (a.corner 1) -- node[above]{j} (a.corner 5);
\node[right] at (a.side 16) {$J$};
\draw[thick, dashed, blue] (a.corner 18) -- node[left]{J} (a.corner 14); 
\end{tikzpicture}
\begin{tikzpicture}[scale=2]
\node [draw, minimum size=5cm,regular polygon,regular polygon sides=18] (a) {};
\node[below] at (a.side 10) {$0$};
\node[left] at (a.side 5) {$i$};
\node[left] at (a.side 6) {$i_1$};
\node[right] at (a.side 14) {$k$};
\node[below] at (a.side 11) {$i_2$};
\draw[very thick] (a.corner 7) -- node[above] {i} (a.corner 11);
\node[right] at (a.side 13) {$k_2$};
\node[above] at (a.side 18) {$k_1$};
\draw[very thick] (a.corner 1) -- node[left] {k} (a.corner 13);
\node[right] at (a.side 16) {$J$};
\draw[thick, dashed, blue] (a.corner 18) -- node[left]{J} (a.corner 14); 
\node[above] at (a.side 3) {$j$};
\draw[thick, red] (a.corner 2) -- node[above]{j} (a.corner 6);
\end{tikzpicture}
\caption{One the left the case $i<j<k$ and $i\lhd k$. On the right $i<j<k$  and $k \lhd i$.}\label{fig2}
  \end{figure}

\begin{lemma}
Let $T$ be a non-crossing tree in a based regular $n+1$-gon. Then the Hasse diagram of the interval-poset $([1,n],\lhd_{T})$ does not contain any configuration of the form $y\to z$ and $y\to x$ where $x<y<z$. 
\end{lemma}
\begin{proof}
Let us assume that we have integers $x < y < z$ such that $y\lhd x$ and $y \lhd z$. This means that the edge $x = [x_1,x_2]$ separates $y = [y_1,y_2]$ from the base. As in the proof of Lemma \ref{lem1}, this implies that 
$$ x_1 \leqslant y_1 \leqslant y_2 \leqslant x_2.$$ 
Similarly, the edge $z = [z_1,z_2]$ separates $y$ from the base, so we have
$$ z_1 \leqslant y_1 \leqslant y_2 \leqslant z_2.$$
Since $T$ is a non-crossing tree, if $x_1 \leqslant z_1$, then necessarily $x_2 \geqslant z_2$. In this case the edge $x$ separates the edge $z$ from the base and we have $z \lhd x$. If $z_1 \leqslant x_1$, then $x_2 \leqslant z_2$ and we have $x\lhd z$. In both cases, we see that one of two relations $y \lhd x$ and $y\lhd z$ is not a cover relation. In particular the configuration $y\to z$ and $y\to x$ does not appear in the Hasse diagram of the poset. 
\end{proof}
\begin{definition}\label{exc}
An interval-poset whose Hasse diagram does not contain any configuration of the form $y\to z$ and $y\to x$ where $x<y<z$ is called an \emph{exceptional} interval-poset.
\end{definition}
If $(P,\lhd)$ is an interval-poset over the integers $[1,n]$ we can construct a graph $G_P$ in a \emph{based} regular $n+1$-gon by using the following procedure which is nothing but a reformulation in terms of interval-posets of the construction explained in Section $5.1$ of \cite{chapoton_etal_mould}. Let us start by labelling the boundary edges of the polygon as above. Then for an integer $v$ consider the poset $\{ x\in [1,n]\ ;\ x\lhd v\}$. This poset has a minimal element (for the usual order relation $<$) $v_1$ and a maximal element $v_2$. We associate to $v$ the edge in the polygon from the left side of $v_1$ to the right side of $v_2$.
\begin{lemma}\label{lem2}
If $(P,\lhd)$ is an exceptional interval-poset on the integers $[1,n]$, then the graph $G_P$ is a noncrossing tree. 
\end{lemma}
\begin{proof}
Let $(P,\lhd)$ be an exceptional interval-poset. If $k$ is a maximal element of $P$ (for the relation $\lhd$), then the set $I_k := \{ i \in P\ ;\ i\lhd k\}$ is an interval because $P$ is an interval-poset. Moreover, if $k$ and $k'$ are two maximal elements of $P$, then the intervals $I_k$ and $I_k'$ are disjoint. Indeed, let $z\in P$ such that $z\lhd k$ and $z\lhd k'$. We can assume that $k\leqslant k'$. If $z \leqslant k \leqslant k'$, then the interval-poset condition implies that $k\lhd k'$ and by maximality $k=k'$. Similarly, if $k\leqslant k' \leqslant z$, the interval-poset condition implies that $k' \lhd k$ and by maximality, we have $k=k'$. Now, if $k < z <k'$, by maximality of $k$ and $k'$, we have a configuration of the form $z \to k$ and $z\to k'$ in the Hasse diagram. This is not possible since the interval poset $P$ is exceptional. In other words, the exceptional interval-posets are nothing but the \emph{non-interleaving forests} introduced in Section $5.1$ of \cite{chapoton_etal_mould}. In particular, the result is a direct consequence of Lemma $5.2$ \cite{chapoton_etal_mould}. We sketch it for the convenience of the reader. 

It is easy to see that the poset $P$ has a unique maximal element if and only if the base of the polygon is in the graph $G_P$. In this case, we say that $G_P$ is based. 

The interval-poset $P$ is disjoint union of $s$ interval-posets $I_{k_1},\cdots, I_{k_n}$ where $k_i$ runs through the maximal elements of $P$. If there is more than one maximal element, by induction on the size of the poset we have that the graph $G_{I_{k_i}}$ is a based noncrossing tree. Now, it is easy to see that the graph $G_P$ is obtained by gluing the base of all the noncrossing trees $G_{I_{k_i}}$ on the boundary of a regular $s+1$-gon. More formally, in terms of NCP-operads, we have $G_p = S \circ_{1} G_{I_{k_1}} \circ_{2} \cdots \circ_{s} G_{I_{k_s}}$, where $S$ is the noncrossing tree with $s$ edges consisting of all boundary edges of the regular $s+1$-gon, except for the base. 

If there is only one maximal element $m$ in $P$, then $G_P$ is based. The case where $P$ has only two elements is elementary and can be checked by listing all the possible cases. If $|P|\geqslant 3$, let $P_{1} = \{ i \in P\ ;\ i < m \}$ and $P_{2} = \{ i \in P\ ;\ m < i\}$. Clearly $P_1$ and $P_2$ are two disjoint interval-posets of size smaller than $|P|$. By induction, the graphs $I_{P_1}$ and $I_{P_2}$ are noncrossing trees. Let $U$ be the noncrossing tree in a based square consisting of all the base and the two adjacent boundary edges. It is now easy to see that $G_P = (U\circ_1 I_{P_1}) \circ_3 I_{P_2}$. In particular, $G_P$ is a noncrossing tree. 
\end{proof}
\begin{proposition}
The map sending a noncrossing tree $T$ to the interval-poset $P_T$ and the map sending an exceptional interval-poset $P$ to the noncrossing tree $T_P$ are two bijections inverse from each other between the set of noncrossing trees in a based regular $n+1$-gon and the set of exceptional interval-posets of size $n$.
\end{proposition}
\begin{proof}
The result is proved by induction. The cases $n=0,1$ and $2$ can be easily checked by hand. Let $n\geqslant 3$. If $T$ is a noncrossing tree, we denote by $P_T$ the exceptional interval-poset obtained in Lemma \ref{lem1}. If $P$ is an exceptional interval-poset, we denote by $T_P$, the noncrossing tree obtained in Lemma \ref{lem2}. Let $S$ be the noncrossing tree with $s$ edges consisting of all boundary edges of the regular $s+1$-gon, except for the base. Let $T_1,\cdots, T_s$ be $s$ based noncrossing trees. Let $T = S \circ_1 T_1\circ_2 \cdots \circ_s T_s$. The edges of $T_i$ (viewed as edges in $T$) are separated from the base by the base of $T_i$, and the edges of $T_i$ are not separated from the base by any edge of $T_j$ for $i\neq j$. This implies that $P_T$ is the disjoint union of the posets $P_{T_i}$ and all these posets have a unique maximal element. 

If the poset $P$ has more than one maximal element, we have $P = P_1 \sqcup \cdots \sqcup P_s$ where $P_i$ is the set of elements smaller than the $i$-th maximal element. By the proof of Lemma \ref{lem2}, the corresponding noncrossing tree $T_P$ is of the form $S\circ_{1} I_{P_1}\circ_{2} \cdots \circ_{s} I_{P_s}$. By the remark above, the poset corresponding to the tree $T_P$ is $P_{I_{P_1}} \sqcup \cdots \sqcup P_{I_{P_s}}$. Now, by induction we have that $P_{T_{P}} = P$. 

Similarly, if the tree $T$ is not based, it can be written as $S \circ_1 T_1\circ_2 \cdots \circ_s T_s$ where $T_i$ are based noncrossing trees. So, we have $P_{T} = P_{T_1} \sqcup \cdots \sqcup P_{T_s}$, and $T_{P_{T}} = S \circ_1 T_{P_{T_1}} \circ_2 \cdots \circ_s T_{P_{T_s}}$. One more time, an induction gives the result. 

Let $U$ be the noncrossing tree in a based square consisting of the base and the two adjacent boundary edges. If $T$ is a based noncrossing tree, there are two noncrossing trees $T_1$ and $T_2$ such that $T = U\circ_1 T_1 \circ_3 T_2$. It is easy to see that the poset $P_T$ is of the form $P_1 \sqcup \{m\} \sqcup P_2$, where $m$ is the labelling of the base of $T$, $P_1$ is the subset consisting of the elements smaller (for $<$) than $m$ and $P_2$ is the set of elements bigger than $m$. Since $m$ is the label of the basis it is the unique maximal element of $P_T$. Using this decomposition of based noncrossing trees, and exceptional interval-posets with a unique maximal element, it is easy to prove by induction that $T_{P_T} = T$ and $P_{T_{P}} = P$. 
\end{proof}

By Theorem $5.3$ \cite{chapoton_mould} there is an injective morphism of operads (in the category of sets) $\Theta$ from the operad of noncrossing plants $\mathbf{NCP}$ and the dendriform operad $\mathbf{Dend}$. Using exceptional interval-posets we describe the image of a noncrossing tree by $\Theta$. 

\begin{theorem}\label{image_noncrossing}
Let $T$ be a noncrossing tree. Let the image of $T$ in $\mathbf{Dend}$ be $\sum_{t\in I} t$. Then the set of trees $I$ is the interval of the Tamari lattice corresponding to the exceptional interval-poset $P_T$. 
\end{theorem}
\begin{proof}
Since exceptional interval-posets are the same as non-interleaving forests, the result follows from a reformulation of Section $5.1$ of \cite{chapoton_etal_mould} and a description of interval-posets in terms of linear extension due to Châtel and Pons. We sketch the arguments. 

Let $\phi : \mathbf{NCP} \to \mathbf{Mould}$ be the injection defined in Section $5.2$ \cite{chapoton_mould} or in $5.2$ \cite{chapoton_etal_mould}. Let $\psi : \mathbf{Dend} \to \mathbf{Mould}$ be the injection defined in Theorem $3.1$ of \cite{chapoton_mould}. By Lemma $5.3$ \cite{chapoton_etal_mould}. Since the maps $\Theta$, $\phi$ and $\psi$ are morphisms of operads and since the diagram is commutative on the elements of $\mathbf{NCP}(2)$, the following diagram is commutative.

\[
\xymatrix{
\mathbf{NCP}\ar[rd]^{\phi}\ar[rr]^{\Theta} & & \mathbf{Dend} \ar[dl]_{\psi} \\
& \mathbf{Mould}
}
\]
Moreover, all the morphisms are injective.

Let $T$ be a noncrossing tree. By lemma $5.3$ of \cite{chapoton_etal_mould}, we have $\phi(T) = \sum_{\sigma \in L(P_T)} f_{\sigma}$ where $P_T$ is the exceptional interval-poset that corresponds to $P$ and $L(P_T)$ is the set of all linear extensions of $P_T$ and if $\sigma \in S_n$, then $f_{\sigma}$ is the fraction defined by
\[
f_{\sigma}(u_1,\cdots, u_n) = \frac{1}{u_{\sigma(1)} \cdot (u_{\sigma(1)} + u_{\sigma(2)} ) \cdot \cdots \cdot (u_{\sigma(1)} + \cdots + u_{\sigma(n)})}. 
\]
For $\sigma,\sigma' \in S_n$ the multi-residue $\oint_{\sigma}$ (see Proposition $3.3$ \cite{chapoton_mould}) have the property that $\oint_{\sigma} f_{\sigma'} \neq 0$ if and only if $\sigma = \sigma'$. So for $\sigma \in S_{n}$, we have $\oint_{\sigma} \phi(P_T) \neq 0$ if and only if $\sigma$ is a linear extension of $P_T$. 

On the other hands, by Proposition $3.3$ of \cite{chapoton_mould}, if $T$ is a binary tree, we have $\oint_{\sigma} \psi(T) \neq 0$ if and only if $\sigma$ is a linear extension of the poset induced by the tree $T$. As consequence, $I$ is the set of trees whose linear extensions are exactly the linear extensions of $P_T$. Now, by Theorem $2.8$ of \cite{pons_chatel}, this implies that $I$ is an interval of the Tamari lattice, and that $P_T$ is the interval-poset corresponding to $I$. 
\end{proof}
\subsection{Noncrossing partitions}
A partition $(b_1,\cdots,b_n)$ of $\{1,\cdots, n\}$ is \emph{noncrossing} if there do not exist $1\leqslant i < j < k < l \leqslant n$ such that $i,k \in b_s$ and $j,l \in b_t$ for $s\neq t$. Let $\NC_n$ be the set of all noncrossing partitions of $\{1,\cdots, n\}$. It is well-known that the cardinality of this set is the Catalan number $c_n$. The refinement of partitions induces a structure of partial order on $\NC_n$ which is known to be a lattice (see \cite{kreweras} for more details). 

It is also classical that the Tamari lattice is a refinement of the poset of noncrossing partitions. In general, it is convenient to realize these posets on the set of Dyck paths via well chosen bijections in order to compare them (see Section $2$ of \cite{bernardi_bonichon} for more details). Here, in order to simplify the proofs, we will realize the poset of noncrossing partions on the Tamari lattice, using a bijection similar to a bijection introduced by Edelman \cite{edelman}.

If $T$ is a (planar) binary tree, we can view it as a binary search tree using the in-order algorithm (this is why our bijection is not the same as Edelman's bijection: he labelled the trees with the pre-order traversal).  Then, the partition $\pi_T$ associated to the tree $T$ is the finest partition of $\{1,2,\cdots, n\}$ such that if $j$ is right child of $i$, then $i$ and $j$ are in the same block. For example, the partition corresponding to the binary tree of Figure \ref{fig_inorder} is $\{1,3,4\}, \{2\}, \{5,8\}, \{6,7\}$. 
\begin{lemma}
Let $T$ be a binary tree and $\pi_T$ its corresponding partition. Then, $\pi_T$ is a noncrossing partition.
\end{lemma}
\begin{proof}
Let $i< j < k < l$ such that $i,k$ are in a block $b_1$ and $j,l$ are in a block $b_2$. The vertex of $T$ labelled by $k$ is a right descendant of the vertex labelled by $i$. 

Since the in-order algorithm goes first through left subtree, then it visits the root and finally goes through right subtree the vertex $j$ is in the right subtree of $i$. Since $l$ and $i$ are in the same block, the vertex $l$ is right-descendant of $i$. Since $k<l$, the vertex $k$ is in the right subtree of $j$. The only possibility is that $j,k$ and $l$ are right descendants of $i$. So, they are in the same block.
\end{proof}
Conversely, if $\pi = (b_1,\cdots, b_n)$ is a noncrossing partition of $\{1,\cdots, n\}$ we will construct a binary search tree associated to this partition. We assume that the blocks of the partition are totally ordered in such a way that $\min(b_1) < \min(b_2) <\cdots < \min(b_n)$ and the elements of the blocks are ordered by the natural order of the integers. The tree $T_\pi$ is constructed in two steps:
\begin{enumerate}
\item To each block $b_i$ is associated a binary tree $T_i$ with root $\min(b_i)$ and if $y$ is the successor of $x$ in the block $b_i$ then, $y$ is the right son of $x$. 

\item Then, if $T_i$ is a tree constructed in the first step, let $m_i$ be the vertex with maximal labelling in the tree. We construct inductively a tree $T_\pi$ by grafting the root of $T_i$ as the left son of the vertex labelled by $m_i + 1$. For an example see Figure \ref{bij_ncp}.
\end{enumerate}
\begin{figure}[h]
\centering
    \begin{tikzpicture}[scale =0.5]
\node (1) at (-4,4) {1}; \node (2) at (-3,4) {2}; \node (3) at (-2,4) {3}; \node (4) at (-1,4) {4}; \node (5) at (0,4) {5}; \node (6) at (1,4) {6}; \node (7) at (2,4) {7}; \node (8) at (3,4) {8};
\draw[blue,thick] (2.north) to  [out=140,in=50] (1.north);
\draw[blue,thick] (4.north) to  [out=100,in=80] (3.north);
\draw[blue,thick] (7.north) to  [out=135,in=45] (2.north);
\draw[blue,thick] (6.north) to  [out=100,in=80] (5.north);
\end{tikzpicture}
\begin{tikzpicture}
\node (1) at (2,0) {1}; \node (2) at (2.5,0.5) {2} ; \node (7) at (3,1) {7}; \node (3) at (0,1) {3}; \node (4) at (0.5,1.5) {4}; \node (5) at (1,0.5) {5}; \node (6) at (1.5,1) {6}; \node (8) at (3.5,-1) {8};
\draw (1)--(2)--(7)
	 (3)--(4)
	 (5)--(6);
\draw[red,->] (1.south) to [out=-45,in=135] (8.west);
\draw[red,->] (3.south) to [out = -45, in= 135] (5.west);
\draw[red, ->] (5.south) to [out = -45, in = 135] (7.west);
\end{tikzpicture}
\begin{tikzpicture}[scale =0.8]
\node (a) at (0,0) {8};
\node (c) at (-1,3) {3};
\node (b) at (-0.5,0.5) {1};
\node (d) at (0,1) {2};
\node (e) at (-0.5, 3.5) {4};
\node (f) at (0, 2) {5};
\node (g) at (1.5,3.5) {6} ; 
\node (h) at (0.5,1.5) {7};
\node (1) at (-4,4) {}; \node (2) at (-3,4) {}; \node (3) at (-2,4) {}; \node (4) at (-1,4) {}; \node (5) at (0,4) {}; \node (6) at (1,4) {}; \node (7) at (2,4) {}; \node (8) at (3,4) {}; \node (9) at (4,4) {}; 
\draw (a)--(b)--(1)
	  (b)--(d)--(h)--(8)
	  (a)--(9)
	  (d)--(2)
	  (h)--(f)--(c)--(3)
	  (f)--(g)--(7)
	  (g)--(6)
	  (c)--(e)--(5)
	  (e)--(4);
\end{tikzpicture}
\caption{An example of the two steps of the construction of a binary tree associated to a noncrossing partition.}\label{bij_ncp}
\end{figure}
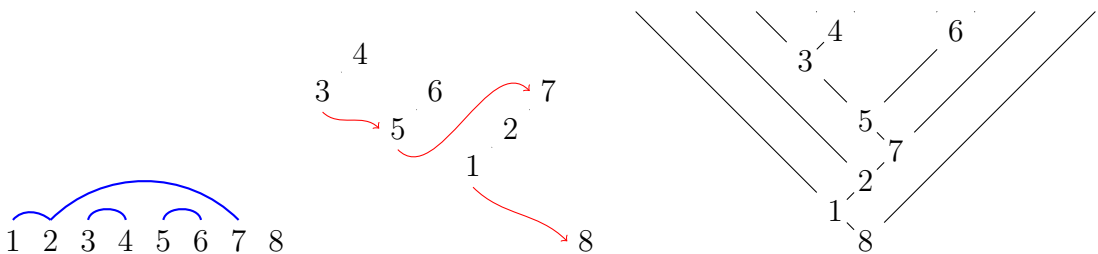

\begin{lemma}
Let $\pi$ be a noncrossing partition of $\{1,\cdots,n\}$ and $T_\pi$ the corresponding binary tree. Then, $T_\pi$ is a binary search tree. 
\end{lemma}
\begin{proof}
Let $s$ be the label of a vertex. If $x$ is a right descendant of $s$, then by construction $s$ and $x$ are in the same block and we have $s < x$. If $y$ is the left son of $x$, then the maximal element of the block of $y$ is $x-1$. This implies that the elements $z$ of the block of $y$ are such that $s<z<x$ because, $s$ and $x$ are in the same block of $y$ and $x-1$ are in the same block, and the partitions are noncrossing. Using these remarks, it is easy to check that if $z$ is in the right subtree of $s$, then $s < z$. Similarly, it is easy to check that the elements of the left subtree of $s$ are labeled by integers strictly smaller than $s$. 
\end{proof}
\begin{proposition}
The map sending a binary tree $T$ to the noncrossing partition $\pi_T$ and the map sending a partition $\pi$ to the binary tree $T_\pi$ are two bijections inverse from each other. 
\end{proposition}
\begin{proof}
By construction of the tree $T$, the minimal elements of the blocks are the vertices  that are left son of another vertex (i.e. they have a right father) and their left descendants are the elements of their block. So, the partition $\pi_{T_\pi}$is equal to $\pi$. Since, there is a unique way to turn a binary tree into a binary search tree of size $n$ using exactly once each of the integers $1,2,\cdots, n$, we have $T_{\pi_T} = T$.  
\end{proof}

We can now be more precise about the fact that the Tamari lattice is a refinement of the lattice of noncrossing partitions.
\begin{lemma}\label{refinement}
Let $\pi_1$ and $\pi_2$ be two noncrossing partitions of $\{1,2,\cdots, n\}$. If $\pi_1 \leqslant \pi_2$ in the poset of noncrossing partitions, then $T_{\pi_{1}} \leqslant T_{\pi_{2}}$ in the Tamari lattice. 
\end{lemma}
\begin{proof}
Using Proposition \ref{char_tamari}, it is enough to show that the decreasing relations of $T_{\pi_1}$ are decreasing relations of $T_{\pi_2}$.

Let $i < j$ such that $j\lhd _{T_{\pi_1}} i$. That is $j$ is in the subtree with root $i$. Since $i<j$, this implies that $j$ is in the right subtree of $i$. Let $x$ be the right descendant of $i$ such that $j$ is in its left subtree (if $j$ is a right descendant of $i$, we have $x=j$). Since the tree $T_{\pi_1}$ is a binary seach tree, this implies that $i< j < x$. Moreover, by construction of $T_{\pi_1}$, the elements $i$ and $x$ are in the same block. Since, the partial order relation for noncrossing partitions is given by merging blocks, in the partition $\pi_2$ the elements $i$ and $x$ are also in the same block. In other words, the element $x$ is in the right subtree of $i$ in $T_{\pi_2}$. Since $i < j < x$, this implies that $j$ is also in the right subtree of $i$, so we have $j \lhd_{T_{\pi_2}} i$.

\end{proof}
We can now characterize the intervals of the Tamari lattice that come from intervals in the lattice of noncrossing partitions.
\begin{theorem}\label{theo_nc}
Let $n \in \mathbb{N}$. Let $I$ be an interval of the Tamari lattice $\Tam_n$. Then, there is an interval of noncrossing partitions $[\pi_1,\pi_2]$ such that $I = [T_{\pi_1},T_{\pi_2}]$ if and only if the interval-poset corresponding to $I$ is exceptional. 
\end{theorem}
\begin{proof}
Let $\pi_1 \leqslant \pi_2$ be two noncrossing partitions. Let $I = [T_{\pi_1},T_{\pi_2}]$ be the corresponding interval in $\Tam_n$ and $P$ be the corresponding interval-poset. Let $x < y < z$ such that we have a relation $y \lhd x$ and $y\lhd z$. 

First assume that $y \lhd x$ is a cover relation. We will show that this imply the existence of a relation $x\lhd z$. This last relation implies that $y\lhd z$ is not a cover relation. We can assume that $y$ is the maximal element such that $y\lhd x$ is a cover relation and $y \lhd z$. Let $t \lhd x$ be a cover relation. If $z \leqslant t$, then by the interval-poset condition we have a relation $z\lhd x$ and the relation $y\lhd x$ becomes the composite of $y \lhd z$ and $z\lhd x$ contradicting the hypothesis. So the maximal element $t$ with a cover relation $t\lhd x$ is an element of $[y,z[$. By the interval-poset condition, we have $t\lhd x$, so by maximality we have $t = y$. 

In other terms, in the decreasing forest of $P$, the element $y$ is the right most child of $x$. So, using the bijection of Theorem \ref{bij} we see that $y$ is the right son of $x$ in the tree $T_{\pi_1}$. In terms of noncrossing partitions, this means that $y$ is the successor of $x$ in its block. Since the partial order relation for the noncrossing partitions is given by merging of blocks, we see that $y$ is still in the block of $x$ in $\pi_2$. This implies that $y$ is also in the right subtree of $x$ in $T_{\pi_2}$.

In the increasing forest of $P$ we have the relation $y \lhd z$ which means that $y$ is in the left subtree of $z$. Since $y$ is a right descendant of $x$, this implies that $x$ is in the right subtree of $z$. Using one more time the bijection of Theorem \ref{bij}, we have an increasing relation $x\lhd z$. 

We only sketch the proof when $y\lhd z$ is a cover relation. We can assume $y$ to be minimal for this property. This implies that $y$ is the most left child of $z$ in the increasing forest of $P$. So $y$ is the left son of $z$ in $T_{\pi_2}$. By the argument of Lemma \ref{refinement}, an increasing relation of $T_{\pi_2}$ is also an increasing relation of $T_{\pi_1}$. In particular, $y$ is in the left subtree of $z$ in $T_{\pi_1}$. The relation $y \lhd x$ in $P$ implies that $y$ is in the right subtree of $x$. Since it is also in the left subtree of $z$, this implies that $z$ is in the right subtree of $x$. So we have the relation $z \lhd x$ in $P$. 

We just proved that the interval-posets of the intervals of the Tamari lattice coming from intervals of non-crossing partitions are exceptional. The result follows from the fact that the number of exceptional interval-posets is the number of intervals in the poset of noncrossing partitions. 
\end{proof}
\section{New intervals and modern interval-posets}
In this section we introduce the notion of \emph{modern} interval-posets and we show that the modern interval-posets of size $n$ are in bijection with the new intervals of $\Tam_{n+1}$. Note that there is a shifting of the size by $1$. 
\subsection{New intervals of the Tamari lattice}
From now on, we will always assume that the leaves of the binary trees of $\Tam_n$ are labeled from left to right by the integers $1,2,\cdots, n+1$. Let $T \in \Tam_n$ and $S\in \Tam_k$. Let $1\leqslant i \leqslant n+1$. The binary tree $T\circ_i S$ is the tree of size $k+n$ obtained by grafting the root of $S$ on the $i$-th leaf of $T$. If $[S_1,T_1]$ is an interval of $\Tam_n$ and $[S_2,T_2]$ is an interval of $\Tam_k$, and $1\leqslant i \leqslant n+1$, then the tree $S_1\circ_i S_2$ is smaller than $T_1\circ_i T_2$. We say that the interval $[S_1 \circ_i S_2,T_1\circ_i T_2]$ is the $i$-th grafting of $[S_2,T_2]$ on $[S_1,T_1]$, and we denote it by $[S_1,T_1] \circ_i [S_2,T_2]$.

\begin{definition}
An interval of $\Tam_n$ is called \emph{new} if it cannot be obtained as the grafting of two intervals. 
\end{definition}
The new intervals were introduced by Chapoton in \cite{chapoton_interval}.

\begin{lemma}[Chapoton]\label{not_new}
An interval $[S,T]$ of $\Tam_n$ is new if and only if there is no pair of subtrees $(A,B)$ of $S$ and $T$ whose leaves are labelled by the same interval $[i,j]\neq [1,n+1]$. 
\end{lemma}
\begin{proof}
If there is a subtree $A$ of $S$ whose leaves are labelled by $[i,j]$ and a subtree $B$ of $T$ whose leaves are also labelled by $[i,j]$, then $S$ is of the form $S_1 \circ_i A$ and $T$ is of the form $T_1 \circ_i B$, so the interval is not new. Conversely, if the interval is not new, then $[S,T] = [S_1,T_1] \circ_i [A,B]$. So there is a pair of subtrees $(A,B)$ of $S$ and $T$ whose leaves are labelled by the same interval $[i,i+\mathrm{size}(S)]$. 
\end{proof}
With this criterion, it is easy to see that the new intervals of $\Tam_n$ have a \emph{nice shape}.
\begin{lemma}\label{new_shape}
Let $n\leqslant 1$. Let $[S,T]$ be a new interval of $\Tam_n$. Then, there are two binary trees $S_1$ and $T_1$ in $\Tam_{n-1}$ such that $S = Y\circ_1 S_1$ and $T = Y\circ_2 T_1$ where $Y$ is the unique binary tree of size $1$. 
\end{lemma}
\begin{proof}
The covering relation for the Tamari lattice is the left rotation. So if there is a vertex on the right side of $S$, it will be fixed by any left rotation, so it will also appear at the same place in the tree $T$. Similarly, if there is a vertex on the left side of $T$ it must also be at the same place in $S$. So the subtrees with root $s$ have the same interval of leaves. Using Lemma \ref{not_new}, we see that the interval $[S,T]$ is not new in both cases.
\end{proof}
However, it is easy to see that there are some intervals with this nice shape but which are not new. We will characterize the new intervals in this family in Theorem \ref{new_chara}.
\subsection{Rising and falling interval-posets}
\begin{definition}\label{def_modern}
Let $n \in \mathbb{N}$. An interval-poset of size $n$ is modern if it does not contain any configuration of the form $x\lhd y$ and $z\lhd y$ with $x < y < z$. 
\end{definition}
Let us remark that unlike Definition \ref{exc}, the forbidden configuration here involves \emph{all the relations} and not only the relations in the Hasse diagram of the poset.

Let us introduce the \emph{rise} of a set with a reflexive binary relation\footnote{The rise of an interval-poset needs not to be an interval-poset, so in order to be able to take successive rises we need a more general setting.}. If $P = \{1,2,\cdots, n\}$ is a set with a reflexive binary relation $\lhd$, then $(\Ri(P),\lhd_R)$ is the set $\{1,2,\cdots, n+1\}$ with the binary relation $\lhd_R$ defined by keeping all decreasing relations of $P$ and shifting by $1$ all the increasing relations of $P$. More precisely, the relation $\lhd_R$ is reflexive and for $x<y \leqslant n$, we have $ y \lhd_R x $ if and only if $y\lhd x$. For $1 < x < y \leqslant n+1$ we have $x\lhd_R y$ if and only if $x-1 \lhd y-1$. For an example, see Figure \ref{ex_rise}. 

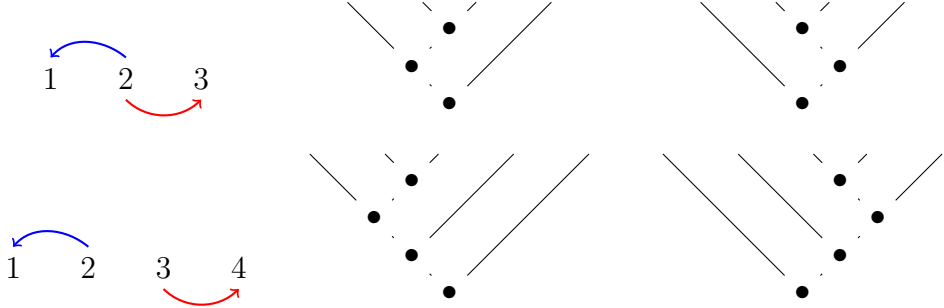
\begin{figure}[h]
\begin{tabular}{ccc}

\begin{tikzpicture}
\node (1) at (0,0) {1}; \node (2) at (1,0) {2}; \node (3) at (2,0) {3};
\draw[blue,thick,->] (2.north) to  [out=140,in=50] (1.north);
\draw[red,->,below, thick] (2.south) to [out=-45,in=225] (3.south);
\end{tikzpicture}
&
\begin{tikzpicture}
\node (1) at (-1,1) {} ; \node (2) at (0,1) {}; \node (3) at (1,1) {} ; \node (4) at (2,1) {}; \node (5) at (0.5,-0.5) {$\bullet$}; \node (6) at (0,0) {$\bullet$}; \node (7) at (0.5,0.5) {$\bullet$};
\draw (1)--(6)--(5)--(4) 
	 (6)--(7)--(2)
	 (7)--(3);
\end{tikzpicture}
&
\begin{tikzpicture}
\node (1) at (-1,1) {} ; \node (2) at (0,1) {}; \node (3) at (1,1) {} ; \node (4) at (2,1) {}; \node (5) at (0.5,-0.5) {$\bullet$}; \node (6) at (1,0) {$\bullet$}; \node (7) at (0.5,0.5) {$\bullet$}; 
\draw (1)--(5)--(6)--(4)
	  (6)--(7)--(2)
	  (7)--(3); 	 
\end{tikzpicture}
\\
\begin{tikzpicture}
\node (1) at (0,0) {1}; \node (2) at (1,0) {2}; \node (3) at (2,0) {3}; \node (4) at (3,0) {4};
\draw[blue,thick,->] (2.north) to  [out=140,in=50] (1.north);
\draw[red,->,below, thick] (3.south) to [out=-45,in=225] (4.south);
\end{tikzpicture}
&
\begin{tikzpicture}
\node (1) at (-2,1) {} ; \node (2) at (-1,1) {}; \node (3) at (0,1) {} ; \node (4) at (1,1) {}; \node (5) at (2,1) {}; \node (r) at (0,-1) {$\bullet$}; \node (a) at (-0.5,-0.5) {$\bullet$}; \node (b) at (-1,0) {$\bullet$}; \node (c) at (-0.5,0.5) {$\bullet$};
\draw (1)--(b)--(a)--(r)--(5)
	   (b)--(c)--(3)
	   (c)--(2)
	  (a)--(4);
\end{tikzpicture}
&
\begin{tikzpicture}
\node (1) at (-2,1) {} ; \node (2) at (-1,1) {}; \node (3) at (0,1) {} ; \node (4) at (1,1) {}; \node (5) at (2,1) {}; \node (r) at (0,-1) {$\bullet$}; \node (a) at (0.5,-0.5) {$\bullet$}; \node (b) at (1,0) {$\bullet$}; \node (c) at (0.5,0.5) {$\bullet$};
\draw (1)--(r)--(a)--(b)--(5)
	   (b)--(c)--(4)
	   (c)--(3)
	  (a)--(2);
\end{tikzpicture}
\end{tabular}
\caption{One the top an interval-poset of size $3$ and its corresponding interval of $\Tam_3$. On the bottow, the rise of this interval-poset and the corresponding interval of $\Tam_4$.}\label{ex_rise}
\end{figure}
\begin{lemma}\label{montee}
Let $(P,\lhd)$ be an interval-poset of size $n$. Then, the rise of $P$ is an interval-poset if and only if $P$ is modern. 
\end{lemma}
\begin{proof}
Since the rising operation only shifts the increasing relations of $P$, it is clear that the relation $\lhd_R$ satisfies the two conditions of interval-poset.

If the interval-poset $P$ is not modern, there is a configuration of the form $x\lhd y$ and $z\lhd y$ with $x < y < z$. The condition of interval-poset implies the existence of the two relations $y-1 \lhd y$ and $y+1\lhd y$. It is clear that the rise of $P$ is not a poset since we have $y\lhd_R y+1$ and $y+1 \lhd_R y$. 

If the interval-poset $P$ is modern, we need to see that $\Ri(P)$ is a poset. If we have in $\Ri(P)$ two elements $x<y$ such that $x\lhd_R y$ and $y\lhd_R x$. Then, in $P$ we have $y\lhd x$ and $x-1 \lhd y-1$. Since $x-1 < x \leqslant y-1$, the condition of interval-poset of $P$ implies that we have a relation $x \lhd y-1$. Similarly since $x\leqslant y-1 < y$, the interval-poset condition implies that we have a relation $y-1 \lhd x$. Since $P$ is a poset, we have $x = y-1$, and we see that the relations $x\lhd_R y$ and $y\lhd_R x$ come from the relations $y-2 \lhd y-1$ and $y \lhd y-1$ in $P$. In other words, the interval-poset $P$ is not modern.

Let us assume that $\Ri(P)$ contains two relations $x\lhd_R y$ and $y\lhd_R z$ but does not contain the relation $x \lhd_R z$. Since increasing relations and decreasing relations come from $P$, it is clear that such a situation implies that one of the two relations is increasing, and the second one is decreasing. If the relation $x \lhd_R y$ is increasing, there are two possibilities: either $z$ is before $x$, or $z$ is between $x$ and $y$. If $z$ is before $x$ the relation $y\lhd_R z$ and the interval-poset condition imply the existence of a relation $x\lhd_R z$. Otherwise, the interval-poset condition implies the existence of a relation $z \lhd_R y$, which by the argument above implies that $P$ is not modern. The case where $x\lhd_R y$ is decreasing is similar. 

\end{proof}
\begin{definition}\label{new_ip}
An interval-poset $P$ of size $n$ is called \emph{new} if it has no increasing relation starting at $1$, no decreasing relation starting at $n$ and no relations of the form $i+1 \lhd_P j+1$ and $j\lhd_P i$ for $i < j$.
\end{definition}
Let us define the \emph{fall}  of an interval poset $(P,\lhd)$ of size $n$ with no increasing relation starting at $1$ and no decreasing relation starting at $n$. This is the poset $(\Fa(P),\lhd_F)$ where $\Fa(P)$ is the set $\{1,2,\cdots, n-1\}$ and the relation $\lhd_F$ is the relation obtained by keeping the decreasing relations and shifting by $-1$ the increasing relations. More precisely, $\lhd_F$ is reflexive and for $x<y$, we have $y\lhd_F x$ if and only if $y\lhd x$ and $x \lhd_F y$ if and only if $x+1 \lhd y+1$. 

\begin{lemma}
Let $P$ be an interval-poset of size $n$ with no increasing relation starting at $1$ and no decreasing relation starting at $n$. Then the fall of $P$ is an interval-poset if and only if $P$ is new. 
\end{lemma}
\begin{proof}
This is a straightforward checking. 
%
%
%
\end{proof}
\begin{lemma}\label{r/f}
The rising/falling operations induce a bijection between the set of modern interval-posets of size $n$ and the set of new interval-posets of size $n+1$. 
\end{lemma}
\begin{proof}
The only way to have two relations $i+1 \lhd_R j+1$ and $j\lhd_R i$ for $i < j$ in $\Ri(P)$ is to have $i \lhd j$ and $j\lhd i$ in $P$, so the rise of a modern interval-poset is new. Similarly, the fall of a new interval-poset $P$ is modern since the forbidden pattern leads to the existence of relations $y-1 \lhd_F y$ and $y+1 \lhd_F y$ that must come from $y+1 \lhd y$ and $y \lhd y+1$ in $P$. 
 
Moreover, it is obvious that the rising and falling operations are inverse of each other. 
\end{proof}

\begin{proposition}\label{falling}
Let $[S,T]$ be an interval of $\Tam_{n+1}$. Let $P$ be the corresponding interval-poset. Then $P$ is new if and only if there is an interval $[S_1,T_1]$ of $\Tam_{n}$ such that $S = Y\circ_{1} S_1$ and $T=Y \circ_{2} T_1$.
\end{proposition}
\begin{proof}
First we show that there is no increasing relation starting at $1$ in $P$ if and only if there is a tree $T_1$ such that $T=Y \circ_{2} T_1$. Using the left/right symmetry and Lemma \ref{symmetry} we can deduce that there is no decreasing relation starting at $n+1$ in $P$ if and only if there is a tree $S_1$ such that $S = Y\circ_{1} S_1$. If there is an increasing relation starting at $1$, let $x$ the maximal element such that we have $1\lhd x$. Then, in the forest of increasing relations the first tree has root $x$ and $1$ is in this tree. So it is sent by the bijection of Theorem \ref{bij} to the binary tree $T$ which has a root $x$ and $1$ is in its left subtree. This implies that the root of $T$ has a left son and is not of the form $Y \circ_{2} T_1$. Conversely, if the root of $T$ has a left son, then the vertex labelled by $1$ is in the left subtree of $T$. Let $x$ be the label of the root of $T$. Then, we have an increasing relation $1\lhd x$ in $P$. 

If $P$ is new, then by the previous argument $S = Y\circ_{1} S_1$ and $T=Y \circ_{2} T_1$. Since $P$ is new, the fall of $P$ is defined. We will show that the interval corresponding to $\Fa(P)$ is $[S_1,T_1]$. Using the left/right symmetry and Lemma \ref{symmetry}, it is enough to show that the binary tree corresponding to the decreasing relations of $\Fa(P)$ is $S_1$. If $F$ denotes the forest of decreasing relations of $\Fa(P)$, then the decreasing forest of $P$ is $F\sqcup \{n+1\}$ where $n+1$ is the tree with only one vertex $n+1$. So, the tree corresponding to the decreasing relations of $\Fa(P)$ is the left subtree of the tree of $P$. In other words, it is the tree $S_1$.  

Since $\Fa(P)$ is an interval-poset, the trees $S_1$ and $T_1$ obtained by considering the decreasing relations and the increasing relations satisfy $S_1\leqslant T_1$ in $\Tam_n$. 

Conversely, if $[S_1,T_1]$ is an interval of $\Tam_{n}$ such that $[S,T]$ is an interval of $\Tam_{n+1}$ for $S = Y\circ_{1} S_1$ and $T=Y \circ_{2} T_1$. If we turn $T$ into a binary search tree by using the in-order algorithm, it is easy to see that the root of $T$ is labelled by $1$ and if $x$ is the label of a vertex of $T_1$, then this vertex is labelled by $1+x$ in $T$. In other words, the increasing relations of $T$ are the increasing relations of $T_1$ shifted by $1$. By symmetry we have that the interval-poset corresponding to $[S,T]$ is the rise of the interval-poset corresponding to $[S_1,T_1]$. By Lemma \ref{r/f}, the interval-poset corresponding to $[S,T]$ is new.  
\end{proof}
\subsection{Characterization of the new intervals}
This section is devoted to the proof of the following Theorem.
\begin{theorem}\label{not_new_theo}
An interval of the Tamari lattice is new if and only if the corresponding interval poset is new. 
\end{theorem}
We are going to prove that the intervals that are not new are exactly the intervals whose interval-poset is not new. As first easy case, we consider intervals that don't have the nice shape of Lemma \ref{new_shape}
\begin{lemma}\label{red1}
Let $n\in \mathbb{N}^{*}$. Let $[T,S]$ be an interval of $\Tam_n$ that is not of the form $[Y\circ_1 S_1, Y\circ_2 T_1]$ for $S_1$ and $T_1$ two trees of $\Tam_{n-1}$. The the corresponding interval-poset is not new.
\end{lemma}
\begin{proof}
If the root of the tree $S$ has a right son. Let $x$ be the most right vertex of $S$. This is the last right descendant of the root of $S$. This vertex is the last vertex visited by the in-order algorithm describe in Section \ref{section1}. So it is labeled by $n$. Let $r$ be the label of the root of $S$. Then, in $P$ we have a relation $n\lhd r$. So the poset is not new. Similarly, if the root of $T$ has a left son, there is an increasing relation in $P$ starting at $1$, so the poset is not new. 
\end{proof}
Similarly, we have
\begin{lemma}\label{red2}
Let $P$ be an interval-poset. If there is an increasing relation starting at $1$ or a decreasing relation starting at $n$, then the corresponding interval is not new. 
\end{lemma}
\begin{proof}
If there is a decreasing relation starting by $n$ in $P$, it means that in the decreasing forest of $P$, the integer $n$ is not the root of its tree. Using the bijection of Theorem \ref{bij}, we see that this implies that there is a vertex on the right side of the tree $S$. Similarly, if there is an increasing relation starting by $1$ in $P$, there is a vertex on the left side of the tree $T$. By Lemma \ref{new_shape}, this implies that the interval $[S,T]$ is not new. 

\end{proof}

With the in-order algorithm, there is a simple link between the labelling of the vertices and the labelling of the leaves.

\begin{lemma}\label{label_subtree}
Let $S$ be a binary search tree. Let $T$ be a subtree of $S$. Then, the vertices of $T$ are labelled by the interval $[i,j-1]$ if and only if the leaves of $T$ are labelled by $[i,j]$.
\end{lemma}

\begin{proof}
The result follows from an easy induction. 

%
\end{proof}

We can deduce the following Lemma,
\begin{lemma}\label{imp1}
Let $[S,T]$ be an interval of $\Tam_n$ such that $S = Y\circ_1 S_1$ and $T = Y\circ_2 T_1$ for $S_1$ and $T_1$ two trees of $\Tam_{n-1}$. If $[S,T]$ is not new, the corresponding interval-poset is not new. 
\end{lemma}
\begin{proof}
By Lemma \ref{not_new}, there are integers $1<i<j<n+1$, a subtree $A$ of $S$ whose leaves are labeled by $[i,j]$ and a subtree $B$ of $T$ whose leaves are also labeled by $[i,j]$. This implies that the root of $A$ and $B$ are not on the left or right sides of $S$ and $T$. 

By Lemma \ref{label_subtree}, the vertices of the two subtrees are labelled by $[i,j-1]$. Let $x$ be the label of the root $B$. The most left vertex of $B$ is labelled by $i$. So, in the poset of increasing relations of $T$ we have $i\lhd x$. The vertex labelled by $j$ (there is such a vertex since $j<n+1$) is the vertex visited by the in-order traversal after $j-1$. Since $j-1$ is the most-right vertex of the tree $B$, the vertex $x$ is in the subtree with root $j$. So we have $x\lhd j$ and by transitivity, we have $i\lhd j$.

Similarly if $y$ is the label of the root of $A$, then we have a decreasing relation $j-1\lhd y$. The vertex labelled by $i-1$ (there is such a vertex since $1<i$) is the vertex visited by the in-order algorithm before the vertex labelled by $i$ which is nothing but the left most vertex of the tree $S_1$. In particular, $y$ is in the subtree with root $i-1$. So we have $y\lhd i-1$ and by transitivity $j-1\lhd i-1$. 

In conclusion, the interval-poset corresponding to $[S,T]$ is not new. 
\end{proof}

Conversely, we need to understand how the forbidden configuration of Definition \ref{new_ip} leads to the existence of a grafting decomposition of the corresponding interval. For this we need to carefully follow the bijection of Châtel and Pons. 

Let $P$ be an interval-poset with no increasing relation starting at $1$ and no decreasing relation starting at $n$. If $P$ is not new, then it has a configuration of the form $i+1 \lhd_R j+1$ and $j\lhd_R i$ for $i < j$. Let $x$ be the maximal element in $[i+1,j]$ such that $i+1 \lhd x$. Note that the interval-poset condition implies that there is a decreasing relation $x\lhd i$. Similarly, let $y$ be the minimal element such that $i< y \leqslant j$ and such that $j \lhd y$. 

\begin{lemma}\label{croissante}
Let $T$ be the upper bound of the interval of $\Tam_n$ corresponding to $P$ by the bijection of Theorem \ref{bij}. Then, the subtree of $T$ with root the vertex labelled by $x$ has leaves labelled by the interval $[i+1,j+1]$. 
\end{lemma}
\begin{proof}
Let $h \leqslant i$. If there is a relation $h \lhd x$, by the interval-poset condition we have a relation $i \lhd x$. This contradicts the decreasing relation $x\lhd i$. 

Moreover, the maximality of $x$ implies that the relation $x\lhd j+1$ is a cover relation in the increasing forest of $P$. Together with the previous argument, this shows that $x$ is the left most child of $j+1$ in the increasing forest of $P$. 

The relation $i+1 \lhd j+1$ and the interval-poset condition implies the existence of the relation $j \lhd j+1$. Clearly, $j$ is the right most child of $j+1$ in the increasing forest of $P$. 

So, in the tree $T$, the vertex $j$ is the right most descendant of $x$ and $x$ is the left son of $j+1$. In other words, $j$ is the largest vertex of the subtree with root $x$. Since we have $i+1 \lhd x$, there is a vertex labelled by $i+1$ in the subtree with root $x$. The first argument of the proof implies that this is the smallest vertex of this subtree. So it has its vertices labelled by the interval $[i+1,j]$. Finally, by Lemma \ref{label_subtree} its leaves are labelled by $[i+1,j+1]$.

\end{proof}
Dually, we have a similar result for the decreasing relations.
\begin{lemma}\label{decroissante}
Let $S$ be the lower bound of the interval of $\Tam_n$ corresponding to $P$ by the bijection of Theorem \ref{bij}. 

Then, the subtree of $S$ with root the vertex labelled by $y$ has leaves labelled by the interval $[i+1,j+1]$. 
\end{lemma}
\begin{proof}
This is a straightforward application of Lemma \ref{symmetry} to Lemma \ref{croissante}. 
\end{proof}
\begin{proof}[Proof of Theorem \ref{not_new_theo}] 
By Lemmas \ref{red1} and \ref{imp1} if an interval is not new, then its corresponding interval poset is not new. Conversely, using Lemma \ref{red2}, we can assume that $P$ does not have an increasing relation starting at $1$ nor a decreasing relation starting at $n$. Let $[S,T]$ be the corresponding interval. Then, by Lemmas \ref{croissante} and \ref{decroissante} and the discussion before them, in $S$ and $T$ there are two subtrees whose leaves are labelled by the same interval. Lemma \ref{not_new} implies that $[S,T]$ is not new. 
\end{proof}
As corollary, we also have a characterization in terms of modern-interval posets. 
\begin{corollary}\label{coro_modern}
Let $n$ be an integer. There is a bijection between the set of new-intervals of $\Tam_{n+1}$ and the set of modern interval-posets of size $n$.
\end{corollary}
\begin{proof}
By Theorem \ref{not_new_theo}, an interval of $\Tam_{n+1}$ is new if and only if its corresponding interval-poset is new. By Lemma \ref{r/f}, these interval-posets are in bijection with the modern interval-posets of size $n$. 
\end{proof}
As explain in Lemma \ref{new_shape}, it is easy to see that if an interval $[S,T]$ is new, then $S = Y\circ_{1} S_1$ and $T=Y \circ_{2} T_1$ where $Y$ is the unique binary tree of size $1$. However, this is not a sufficient condition. Using our characterization of new intervals in terms of interval-posets, we can find a characterization of the new intervals of the Tamari lattice.
\begin{theorem}\label{new_chara}
Let $[S,T]$ be an interval of $\Tam_{n+1}$. Then $[S,T]$ is a new interval if and only if there is an interval $[S_1,T_1]$ in $\Tam_n$ such that $S = Y\circ_{1} S_1$ and $T = T\circ_2 T_1$. 
\end{theorem}
\begin{proof}
By Theorem \ref{not_new_theo} the new intervals of $\Tam_{n+1}$ are exactly the intervals such that the corresponding interval-poset is new. The result follows from Proposition \ref{falling}. 
\end{proof}

\section{Infinitely modern interval-posets}
For an integer $k$ and an interval-poset $P$ of size $n$, we let $\Ri^{k}(P)$ the $k$-th rise of $P$. That is the set obtained by successively taking $k$-times its rise.
\begin{definition}
An interval-poset is infinitely modern if $\Ri^{k}(P)$ is an interval-poset for every $k\geqslant 1$. 
\end{definition}
\begin{lemma}\label{inf}
An interval-poset $P$ is infinitely modern if and only if it does not contain any configuration of the form $w\lhd x$ and $z\lhd y$ for $w < x < y < z$. 
\end{lemma}
\begin{proof}
If we have such a configuration in $P$, then the interval-poset condition implies the existence of relations $x-1\lhd x$ and $y+1\lhd y$. After rising our poset enough times, they will lead to $y\lhd_{R^k} y+1$ and $y+1\lhd_{R^k} y$.

Conversely, let $k+1$ the smallest integer such that $\Ri^{k+1}(P)$ is not a poset. Then, $\Ri^{k}(P)$ is not modern, so by Definition \ref{def_modern} there is a configuration of the form $x\lhd_{R^{k}} y$ and $z\lhd_{R^k} y$ for $x < y < z$ in $\Ri^{k}(P)$. This leads to the result. 
\end{proof}
For an interval-poset $P$ of size $n$ we denote by $\ir(P)$ the smallest in integer $k$ such that there is an \underline{i}ncreasing \underline{r}elation $k \lhd k+1$. If there is no increasing relation, we use the convention that $\ir(P)=n$. Similarly, we denote by $\id(P)$ the largest integer $i$ such that there is a \underline{d}ecreasing \underline{r}elation $i\lhd i-1$. If there is no decreasing relation, we use the convention that $\id(P)=1$. We can associate to any interval poset $P$ of size $n$ the double statistic $\big(\ir(P),\id(P)\big)$ which is a pair of elements of $\{1,\cdots,n\}$. Using this statistic, we have another description of the infinitely modern interval-posets.
\begin{proposition}\label{stat}
Let $P$ be an interval-poset of size $n$. Then $P$ is infinitely modern if and only if $\id(P) \leqslant \ir(P)$.
\end{proposition}
\begin{proof}
If $\ir(P) < \id(P)$, then the poset is not infinitely-modern because after some risings, the relation $k\lhd k+1$ will contradict the relation $i\lhd i-1$. Conversely, if the poset is not infinitely modern, by Lemma \ref{inf}, there are integers $w < x < y < z$ such that $w\lhd x$ and $z\lhd y$. By the interval-poset condition, we have relations $x-1\lhd x$ and $y+1\lhd y$. In particular, we see that $\ir(P) < \id(P)$. 
\end{proof}

We denote by $E(n,i,k)$ the set of infinitely modern interval-posets $P$ of size $n$ such that $\ir(P) = k$ and $\id(P)=i$. 

Let $1 \leqslant i \leqslant k \leqslant n+1$ and $P$ be an interval-poset of size $n$. Then, we define a relation $f_{i,k}(P)$ on the set with $n+1$ elements by adding a new point to the set of $P$. For the increasing relations we can think that the new point is inserted at $k$ and we add a new increasing relation from $k$ to $k+1$. The increasing relations of $P$ are shifted by $1$ accordingly to the new point. For the decreasing relations the new point is inserted at the position $i$. A new relation $i\lhd i-1$ is added and the decreasing relations of $P$ are shifted by $-1$ accordingly to the new point. More formally, $f_{i,k}(P)$ is defined as the set $\{1,2,\cdots, n+1\}$ with the relation $\lhd'$: 
\begin{itemize}
\item We have $k \lhd' k+1$ and $i\lhd' i+1$ with the convention that there are no increasing relations when $k=n+1$ and no decreasing relations when $i=1$.
\item Let us assume that we have an increasing relation $x\lhd y$ in $P$. If $x<y<k$, then we have the relation $x\lhd'y$ in $f_{i,k}(P)$. If $x < k \leqslant y$, then we have the relation $x\lhd' y+1$ in $f_{i,k}(P)$.  If $k \leqslant x < y$, then we have the relation $x+1 \lhd' y+1$.
\item Let us assume that we have a decreasing relation $y \lhd x$ in $P$. If $i \leqslant x < y$, then we have the relation $y+1 \lhd'x+1$. If $x< i\leqslant y$, then we have the relation $y+1 \lhd x$. If $x<y <i$, then we have the relation $y\lhd x$. 
\item Take the transitive closure of the relation $\lhd'$. 
\end{itemize}
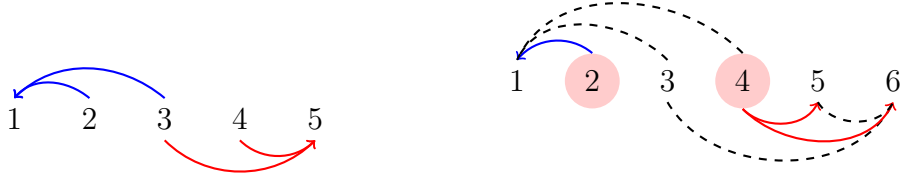
\begin{figure}
\centering
\begin{tikzpicture}
\node (1) at (0,0) {1}; \node (2) at (1,0) {2}; \node (3) at (2,0) {3}; \node (4) at (3,0) {4}; \node (5) at (4,0) {5};
\draw[red,->,below,thick] (4.south) to [out=-45,in=225] (5.south);
\draw[blue,thick,->] (2.north) to  [out=140,in=50] (1.north);
\draw[red,->,below, thick] (3.south) to [out=-45,in=225] (5.south);
\draw[blue,thick,->] (3.north) to [out =140 ,in =45] (1.north);
\end{tikzpicture}
\ \ \ \ \ \ \ \ \ \ \ \ \ \ \ 
\begin{tikzpicture}
\node (1) at (0,0) {1}; \node[fill=red!20, shape=circle] (2) at (1,0) {2}; \node (3) at (2,0) {3}; \node[fill=red!20, shape=circle] (4) at (3,0) {4}; \node (5) at (4,0) {5}; \node (6) at (5,0) {6};
\draw[thick,blue,->] (2.north) to  [out=140,in=50] (1.north);
\draw[dashed,below, thick] (3.south) to [out=-60,in=-115] (6.south);
\draw[red,->,below,thick] (4.south) to [out=-45,in=-115](5.south);
\draw[red,->,below,thick] (4.south) to [out=-45,in=-115](6.south);
\draw[dashed,below, thick] (5.south) to [out=-60,in=-115] (6.south);
\draw[dashed,above,thick] (4.north) to [out=135,in=65](1.north);
\draw[dashed,above,thick] (3.north) to [out=135,in=65](1.north);
\end{tikzpicture}
\caption{On the left, an interval-poset $P$ of size $5$. On the right, the construction $f_{2,4}(P)$. The vertices in red represent the positions of the new arrows which are displayed in thick red and blue. The black dashed corresponds to the relations of $P$. The long red arrow is obtained by transitivity.}
\end{figure}
\begin{lemma}\label{double}
Let $1\leqslant i \leqslant k \leqslant n+1$. Let $i'\leqslant i$ and $k-1\leqslant k'$. Let $P \in E(n,i',k')$. Then, $f_{i,k}(P)$ is an interval-poset of size $n+1$ in $E(n+1,i,k)$. 
\end{lemma}
\begin{proof}
If we have a decreasing relation $ y \lhd x$ in $P$, by the interval poset condition, we have also have a relation $x+1 \lhd x$. This implies that, in $P$ all the decreasing relations are of the form $y\lhd x$ where $x < y$ and $x < i'$. Since $i'\leqslant i$, in $f_{i,k}(P)$ all the decreasing relations are of the form $y'\lhd x'$ where $x' < i$. Moreover, we have a decreasing relation $i \lhd i-1$ in $f_{i,k}(P)$. In other terms, we have $\id(f_{i,k}(P)) = i$.

Similarly, in $P$ all the increasing relations are of the form $x \lhd y$ with $x < y$ and $k'+1 \leqslant y$. Since $k\leqslant k'+1$, all increasing relations in $f_{i,k}(P)$ are of the form $x' \lhd y'$ where $k < y'$. By construction in $f_{i,k}(P)$, we have the relation $k\lhd k+1$. So, $\ir(f_{i,k}(P)) = k$.

It remains to check that under the hypothesis $f_{i,k}(P)$ is an interval-poset. Let $x< y$ such that $x\lhd y$ and $y \lhd x$ in $f_{i,k}(P)$. Since the increasing relations land after $k$ and the decreasing before $i$, the only possibility is to have $x < i < k < y$. This means that in $P$, we have a relation $x \lhd y-1$ and $y\lhd y-1$. This is not possible since $P$ is a poset. Since the relation $\lhd'$ is transitive by construction, this shows that $f_{i,k}(P)$ is a poset. 

We need to check the interval-poset condition. It is an easy case by case checking: let $x<y < z$ and $x\lhd z$ in $f_{i,k}(P)$. If $x < k < y$, then in $P$ we have the relation $x \lhd y-1$. If $k\neq z$, since $P$ is an interval-poset, we have the relation $z' \lhd y-1$ for $z'=z$ if $z'<k$ and $z' = z-1$ otherwise. So in $f_{i,k}(P)$, we have $z\lhd y$. If $z=k$, then in $f_{i,k}(P)$ we have the relation $k\lhd k+1$. By the interval-poset condition of $P$ we have $k\lhd y-1$. It becomes $k+1 \lhd y$ in $f_{i,k}(P)$. By transitivity we have $k\lhd y$. Similarly, we can check the case where $k \leqslant x \leqslant y$. The case of decreasing relations is also similar. 
\end{proof}
On the other hand, if $P$ is an interval-poset in $E(n+1,i,k)$ let us construct $\rho(P)$ an interval-poset of size $n$. Informally, for the increasing relations, we remove the vertex $k$ and the relation $k\lhd k+1$. We shift the other relations accordingly to their position. For the decreasing relations, we remove the vertex $i$ and the relation $i\lhd i-1$. And we shift the relations accordingly to their position. More formally, $\rho(P)$ is the relation on the set $\{1,2,\cdots, n\}$ defined by:
\begin{itemize}
\item Let $x < y$. Then we have a relation $x\lhd y$ in the following two cases: if $x < k < y+1$ and there is a relation $x\lhd y+1$ in $P$, or if $k < x+1 < y+1$ and there is a relation $x+1 \lhd y+1$ in $P$.
\item Let $x < y$. Then we have a relation $y\lhd x$ in the following two cases: if $x  < y  < i$ and there is a relation $y\lhd x$ in $P$ or if $x < i < y+1 $ and there is a relation $ y+1 \lhd x$ in $P$.
\end{itemize}
\begin{lemma}\label{coupe}
Let $P \in E(n+1,i,k)$. Then $\rho(P)$ is an infinitely modern interval-poset such that $\id(P) \leqslant i$ and $k-1 \leqslant \ir(P)$. 
\end{lemma}
\begin{proof}

In $P$ the increasing relations are of the form $y \lhd x$ where $ k < x$. If we have the relation $k-1 \lhd k+1$ in $P$, then we have the relation $k-1 \lhd k$ in $\rho(P)$. Otherwise the second increasing relation $x\lhd x+1$ of length $1$ in $P$ (the one after $k\lhd k+1$) appears for $k+1 \leqslant k$. Here we use one more time the convention that there is an increasing relation starting at $n+1$ if there is no such relation. So in $\rho(P)$ the first increasing relation $x-1\lhd x$. So we have $k-1 \leqslant \ir(P)$, and $\ir(P) = k-1$ if and only if we have the relation $k-1\lhd k+1$ in $P$.

Similarly, we have $\id(P) \leqslant i$ and $\id(P) = i$ if and only if we have the relation $i-1\lhd i+1$ in $P$. 

Now, we check that $\rho(P)$ is an interval-poset. By the description of $\ir(\rho(P))$ and $\id(\rho(P))$, we deduce that if $x\lhd y$ is an increasing relation in $\rho(P)$, we have $k\leqslant y$. Similarly, if $y\lhd x$ is a decreasing relation we have $x < i $. 

Let $x < y$ such that $x\lhd y$ and $y \lhd x$ in $\rho(P)$. Then, we must have $x < i$ and $k\leqslant y$. So, the relation $x\lhd y$ comes from the relation $x \lhd y+1$ in $P$ and the relation $y\lhd x$ comes from the relation $y+1 \lhd x$ in $P$. Since $P$ is an interval-poset, this is not possible. Since in $P$ there are no increasing relations of the form $x\lhd k$ and no decreasing relations of the form $y\lhd i$, removing the relations $k\lhd k+1$ and $i\lhd i-1$ will not break the transitivity of the relation. Checking the interval-poset condition is straightforward and similar to the case of Lemma \ref{double}.

If $i < k$, as direct consequence of Proposition \ref{stat}, the interval-poset $\rho(P)$ is infinitely-modern. If $i = k$, we just have to check that it is not possible to have $\ir(\rho(P)) = k-1$ and $\id(\rho(P)) = i$. But this is a direct consequence of the above description of these two particular cases.

\end{proof}

\begin{proposition}\label{prop_calc}
Let $n \in \mathbb{N}$. Let $1\leqslant i \leqslant k \leqslant n+1$. Then, we have a bijection 
\[ f_{i,k} : \bigcup_{\substack{1\leqslant i' \leqslant i \\ k-1\leqslant k' \leqslant n}} E(n,i',k') \to E(n+1,i,k) .\]
\end{proposition}
\begin{proof}
By Lemma \ref{double} $f_{i,k}$ maps the left hand side to the right hand side, and by Lemma \ref{coupe}, the map $\rho$ goes from the right hand side to the left hand side. It is clear that $\rho$ and $f_{i,k}$ are two bijection inverse from each other. 
\end{proof}
\begin{theorem}\label{theo_inf}
Let $n\in \mathbb{N}$. Then, the number of infinitely modern interval-posets of size $n$ is $\frac{1}{2n+1} {{3n}\choose{n}}$.
\end{theorem}
\begin{proof}
Let $k,l \in \{0,1,\cdots,n-1\}$. We set $B(n,k,l) = |E(n,k+1,n-l)|$. With the change of variables $x -1 = k$ and $ n-y = l$, this is the number of infinitely modern interval-posets of size $n$ with $\ir = y$ and $\id = x$. It is easy to check that we have $B(1,0,0) = 1$. By Lemma \ref{stat}, if $P$ is an interval-poset such that $\ir(P) < \id(P)$, then $P$ is not infinitely-modern. So, if $k+l \geqslant n$, we have $B(n,k,l) = 0$. Finally, if $k+ l < n$, then $1\leqslant x \leqslant y \leqslant n$ and Proposition \ref{prop_calc} implies 
\[B(n,k,l) = \sum_{0\leqslant i \leqslant k, 0\leqslant j \leqslant k}B(n-1,i,j). \]
We recognize the induction formula of Definition $2.1$ of \cite{aval}. The result follows from Proposition $2.1$ \cite{aval}. 
\end{proof}
\bibliographystyle{alpha}

\end{document}